\documentclass[12pt, reqno]{amsart}
\usepackage{amsmath, amsthm, amscd, amsfonts, amssymb, graphicx, color}
\usepackage[bookmarksnumbered, colorlinks, plainpages]{hyperref}
\hypersetup{colorlinks=true,linkcolor=red, anchorcolor=green, citecolor=cyan, urlcolor=red, filecolor=magenta, pdftoolbar=true}

\usepackage{tensor}

\newtheorem{theorem}{Theorem}[section]

\newtheorem{proposition}[theorem]{Proposition}
\newtheorem{corollary}[theorem]{Corollary}
\theoremstyle{definition}
\newtheorem{definition}[theorem]{Definition}

\theoremstyle{remark}
\newtheorem{remark}[theorem]{Remark}
\numberwithin{equation}{section}

\usepackage{fullpage}

\begin{document}

\setcounter{page}{1}

\title[FPT on $G$-metric spaces]{Common fixed points via $\lambda$-sequences in $G$-metric spaces.}

\author[Ya\'e Olatoundji Gaba]{Ya\'e Olatoundji Gaba$^{1,2,*}$}

\address{$^{1}$\'Ecole Normale Sup\'erieure de Natitingou, Universit\'e de Parakou, B\'enin.}

\address{$^{2}$Institut de Math\'ematiques et de Sciences Physiques (IMSP)/UAC,
Porto-Novo, B\'enin.}

\address{$^{*}$\textit{Corresponding author.}}

\email{\textcolor[rgb]{0.00,0.00,0.84}{gabayae2@gmail.com
}}



\begin{abstract}
In this article, we use $\lambda$-sequences to derive common fixed points for a family of self-mappings defined on a complete $G$-metric space. We imitate some existing techniques in our proofs and show that the tools emlyed can be used at a larger scale.
These results generalize well known results in the literature. 
\end{abstract} 

\maketitle

\section{Introduction and preliminaries}
The generalization of the Banach contraction mapping principle has been a heavily investigated branch of research. In recent years, several authors have obtained \textit{fixed} and \textit{common fixed point} results for various
classes of mappings in the setting of many generalized metric spaces. One of them, the $G$-metric space, is our focus in this paper and fixed point results, in this setting, presented by authors like  Abbas\cite{abbas}, Gaba\cite{Gaba1,Gaba5}, Mustafa\cite{Mustafa}, Vetro\cite{v} and many more, are enlighting on the subject. Moreover, in \cite{Gaba4}, we introduced the concept of \textit{$\lambda$-sequence} which extended the idea of \textit{$\alpha$-series} proposed by Vetro et al. in \cite{v}. The present article exclusively presents natural extensions of some results already given by Abbas\cite{abbas} and Vetro\cite{v}, and therefore generalizes some recent results regarding fixed point theory in $G$-metric spaces. We also show how the idea of $\lambda$-sequence are used in proving some of these results.
The method builds on the convergence of an appropriate series of coefficients. We also make use of a special class of homogeneous functions. Recent and similar work
can also be read in \cite{Gaba1,Gaba4,Gaba5,lil}.

\vspace*{0.2cm}

We recall here some key results that will be useful in the rest of this manuscript.
The basic concepts and notations attached to the idea of $G$-metric spaces can be read extensively in \cite{Mustafa} but for the convenience of the reader, we discuss the most important ones.

\begin{definition} (Compare \cite[Definition 3]{Mustafa})
Let $X$ be a nonempty set, and let the function $G:X\times X\times X \to [0,\infty)$ satisfy the following properties:
\begin{itemize}
\item[(G1)] $G(x,y,z)=0$ if $x=y=z$ whenever $x,y,z\in X$;
\item[(G2)] $G(x,x,y)>0$ whenever $x,y\in X$ with $x\neq y$;
\item[(G3)] $G(x,x,y)\leq G(x,y,z) $ whenever $x,y,z\in X$ with $z\neq y$;
\item[(G4)] $G(x,y,z)= G(x,z,y)=G(y,z,x)=\ldots$, (symmetry in all three variables);
\item[(G5)]
$$G(x,y,z) \leq  [G(x,a,a)+G(a,y,z)]$$ for any points $x,y,z,a\in X$.
\end{itemize}
Then $(X,G)$ is called a \textbf{$G$-metric space}.

\end{definition}

The property (G3) is crucial and shall play a key role in our proofs.

\begin{proposition} (Compare \cite[Proposition 6]{Mustafa})
Let $(X,G)$ be a $G$-metric space.
 Then for a sequence $(x_n) \subseteq X$, the following are equivalent
\begin{itemize}
\item[(i)] $(x_n)$ is $G$-convergent to $x\in X.$

\item[(ii)] $\lim_{n,m \to \infty}G(x,x_n,x_m)=0.$


\item[(iii)]$\lim_{n \to \infty}G(x,x_n,x_n)=0.$ 

\item[(iv)]$\lim_{n \to \infty}G(x_n,x,x)=0.$ 
\end{itemize}

\end{proposition}

\begin{proposition}(Compare \cite[Proposition 9]{Mustafa})

In a $G$-metric space $(X,G)$, the following are equivalent
\begin{itemize}
\item[(i)] The sequence $(x_n) \subseteq X$ is $G$-Cauchy.

\item[(ii)] For each $\varepsilon >0$ there exists $N \in \mathbb{N}$ such that $G(x_n,x_m,x_m)< \varepsilon$ for all $m,n\geq N$.

\end{itemize}

\end{proposition}

\begin{definition} (Compare \cite[Definition 9]{Mustafa})
A $G$-metric space $(X,G)$ is said to be complete if every $G$-Cauchy sequence in $(X,G)$ is $G$-convergent in $(X,G)$. 
\end{definition}

\begin{definition} (Compare \cite[Definition 4]{Mustafa})
 A $G$-metric space $(X,G)$ is said to be symmetric if 
 $$G(x,y,y) = G(x,x,y), \ \text{for all } x,y \in X.$$

\end{definition}

\begin{definition}(Compare \cite[Definition 2.1]{Gaba5})
A sequence $(x_n)_{n\geq 1}$ in a metric space $(X,d)$ is a $\lambda$-sequence if there exist $0<\lambda<1$ and $n(\lambda) \in \mathbb{N}$ such that $$\sum_{i=1}^{L-1} d(x_i,x_{i+1}) \leq \lambda L \text{ for each } L\geq n(\lambda)+1.$$

\end{definition}

\begin{definition}(Compare \cite[Definition 6]{Gaba1})
A sequence $(x_n)_{n\geq 1}$ in a $G$-metric space $(X,G)$ is a $\lambda$-sequence if there exist $0<\lambda<1$ and $n(\lambda) \in \mathbb{N}$ such that $$\sum_{i=1}^{L-1} G(x_i,x_{i+1},x_{i+1}) \leq \lambda L \text{ for each } L\geq n(\lambda)+1.$$

\end{definition}

\begin{definition}\label{vetro}(Compare \cite[Definition 2.1]{v})
For a sequence $(a_n)_{n\geq 1}$ of nonnegative real numbers, the series $\sum_{n=1}^{\infty}a_n$ is an $\alpha$-series if there exist $0<\lambda<1$ and $n(\lambda) \in \mathbb{N}$ such that $$\sum_{i=1}^{L} a_i \leq \lambda L \text{ for each } L\geq n(\lambda).$$

\end{definition}

\begin{remark}
For a given $\lambda$-sequence $(x_n)_{n\geq 1}$ in a $G$-metric space $(X,d)$, the sequence $(\beta_n)_{n\geq 1}$ of nonnegative real numbers defined by 
\[
\beta_i = d(x_i,x_{i+1},x_{i+1}),
\]
is an $\alpha$-series.

Moreover, any non-increasing $\lambda$-sequence of elements of $\mathbb{R}^+$ endowed with the $\max\footnote{The max metric $m$ refers to $m(x,y)=\max\{x,y\}$ }$ metric is also an $\alpha$-series.
Therefore, $\lambda$-sequences generalise $\alpha$-series but to ease computations, we shall consider, throughout the paper,  $\alpha$-series\footnote{However, the reader can convince himself that using $\lambda$-sequences do not add to the complexity of the problem.}. 
\end{remark}

\section{First generalizations results}

We begin with the following generalisation of \cite[Theorem 2.1]{v}, the main result of Vetro at al.

Let $\Phi$ be the class of continuous, non-decreasing, sub-additive and homogeneous functions $F:[0,\infty) \to [0,\infty)$ such that $F^{-1}(0)=\{0\}$. 


\begin{theorem}\label{common1}

Let $(X,G)$ be a complete $G$-metric space and $\{T_n\}$ be a family of self mappings on $X$ such that

\begin{align}\label{eq1}
F(G(T_ix,T_jy,T_kz)) \ \leq  &  
  F\left((\tensor*[_{k}]{\Theta}{_{i,j}})\left[G(x,T_ix,T_ix)+\frac{1}{2}[(G(y,T_jy,T_jy)+G(z,T_kz,T_kz)]\right] \right) \nonumber \\
 & + F((\tensor*[_{k}]{\Delta}{_{i,j}})G(x,y,z))
 \end{align}

for all $x,y,z\in X$ with $x\neq y$, $0\leq \tensor*[_{k}]{\Theta}{_{i,j}}, \tensor*[_{k}]{\Delta}{_{i,j}}<1; i,j,k=1,2,\cdots, $ and some $F \in \Phi,$ homogeneous with degree $s$. If 

\[ \sum_{i=1}^{\infty}\left[\frac{[(\tensor*[_{i+2}]{\Theta}{_{i,i+1}})^s+(\tensor*[_{i+2}]{\Delta}{_{i,i+1}})^s]}{1-(\tensor*[_{i+2}]{\Theta}{_{i,i+1}})^s}\right]
\]

is an $\alpha$-series, then $\{T_n\}$ have a unique common fixed point in $X$.

\end{theorem}



\begin{proof}
We will proceed in two main steps.
\vspace{0.3cm}

%
%
%

\underline{Claim 1:} $\{T_n\}_{n\geq 1}$ have a common fixed point in $X$.

\vspace*{0.3cm}

For any $x_0\in X$, we construct the sequence $(x_n)$ by setting $x_n= T_n(x_{n-1}),\ n=1,2,\cdots .$ 

We assume without loss of generality that $ x_m\neq x_{n} $ for all $n \neq m\in \mathbb{N}$.
Using \eqref{eq1}, we obtain, for the triplet $(x_0,x_1,x_2)$,

\begin{align*}
F(G(x_1,x_2,x_3)) & = F(G(T_1x_0,T_2x_1,T_3x_2)) \\ 
               & \leq  
               (\tensor*[_{3}]{\Theta}{_{1,2}})^sF\left(\left[G(x_0,x_1,x_1)+\frac{1}{2}G(x_1,x_2,x_2)+\frac{1}{2}G(x_2,x_3,x_3)\right]\right)\\
               & +(\tensor*[_{3}]{\Delta}{_{1,2}})^sF(G(x_0,x_1,x_2)).       
\end{align*}

\vspace*{0.3cm}

By property (G3) of $G$, one knows that 

\[G(x_i,x_{i+1},x_{i+1}) \leq G(x_{i-1},x_i,x_{i+1}) \ \text{ and }\ G(x_i,x_{i},x_{i+1})\leq G(x_i,x_{i+1},x_{i+2})  .\]

Hence, 
\begin{align*}
F(G(x_1,x_2,x_3)) & = F(G(T_1x_0,T_2x_1,T_3x_2)) \\ \\
 & \leq (\tensor*[_{3}]{\Theta}{_{1,2}})^sF\left(\left[G(x_0,x_1,x_2)+\frac{1}{2}G(x_1,x_2,x_3)+\frac{1}{2}G(x_1,x_2,x_3)\right]\right)\\ 
               & +(\tensor*[_{3}]{\Delta}{_{1,2}})^sF(G(x_0,x_1,x_2))\\ \\
               & = (\tensor*[_{3}]{\Theta}{_{1,2}})^sF(G(x_1,x_2,x_3))+ [ (\tensor*[_{3}]{\Theta}{_{1,2}})^s+(\tensor*[_{3}]{\Delta}{_{1,2}})^s ] F(G(x_1,x_2,x_0))
\end{align*}


i.e. 

\vspace*{0.3cm}

\[
F(G(x_1,x_2,x_3)) \leq \frac{ [ (\tensor*[_{3}]{\Theta}{_{1,2}})^s+(\tensor*[_{3}]{\Delta}{_{1,2}})^s ]}{1-(\tensor*[_{3}]{\Theta}{_{1,2}})^s}F(G(x_0,x_1,x_2)).
\]

Also we get

\begin{align*}
F(G(x_2,x_3,x_4)) & \leq \frac{ [ (\tensor*[_{4}]{\Theta}{_{2,3}})^s+(\tensor*[_{4}]{\Delta}{_{2,3}})^s ]}{1-(\tensor*[_{4}]{\Theta}{_{2,3}})^s}F(G(x_1,x_2,x_3))\\
& \leq \left[  \frac{ [ (\tensor*[_{4}]{\Theta}{_{2,3}})^s+(\tensor*[_{4}]{\Delta}{_{2,3}})^s ]}{1-(\tensor*[_{4}]{\Theta}{_{2,3}})^s} \right]\left[ \frac{ [ (\tensor*[_{3}]{\Theta}{_{1,2}})^s+(\tensor*[_{3}]{\Delta}{_{1,2}})^s ]}{1-(\tensor*[_{3}]{\Theta}{_{1,2}})^s} \right] F(G(x_0,x_1,x_2)).
\end{align*}

Repeating the above reasoning, we obtain

\[
F(G(x_{n},x_{n+1},x_{n+2})) \leq \prod\limits_{i=1}^n \left[\frac{[(\tensor*[_{i+2}]{\Theta}{_{i,i+1}})^s+(\tensor*[_{i+2}]{\Delta}{_{i,i+1}})^s]}{1-(\tensor*[_{i+2}]{\Theta}{_{i,i+1}})^s}\right]F(G(x_0,x_1,x_2)).
\]

If we set 

\vspace*{0.5cm}

$$r_i  =  \left[\frac{[(\tensor*[_{i+2}]{\Theta}{_{i,i+1}})^s+(\tensor*[_{i+2}]{\Delta}{_{i,i+1}})^s]}{1-(\tensor*[_{i+2}]{\Theta}{_{i,i+1}})^s}\right], $$

\vspace*{0.5cm}

we have that 

$$F(G(x_{n},x_{n+1},x_{n+2})) \leq \left[ \prod\limits_{i=1}^n r_i \right] F(G(x_0,x_1,x_2)). $$

\vspace*{0.5cm}

Therefore, for all $l>m>n>2$

\begin{align*}
G(x_n,x_{m},x_{l}) & \leq G(x_{n},x_{n+1},x_{n+1})+  G(x_{n+1},x_{n+2},x_{n+2})  \\
 & + \cdots + G(x_{l-1},x_{l-1},x_{l}) \\
 & \leq G(x_{n},x_{n+1},x_{n+2})+  G(x_{n+1},x_{n+2},x_{n+3})  \\
 & + \cdots + G(x_{l-2},x_{l-1},x_{l}).
\end{align*}


Using the fact that $F$ is sub-additive, we write

\begin{align*}
F(G(x_n,x_{m},x_{l}))& \leq \left(\left[ \prod\limits_{i=1}^n r_i \right] + 
\left[ \prod\limits_{i=1}^{n+1} r_i \right]+ \cdots + \left[ \prod\limits_{i=1}^{l-2} r_i \right]\right) F(G(x_0,x_1,x_2)) \\
& = \sum_{k=0}^{l-n-2}\left[\prod\limits_{i=1}^{n+k}r_i\right]F(G(x_0,x_1,x_2))\\
& =\sum_{k=n}^{l-2}\left[\prod\limits_{i=1}^{k}r_i\right]F(G(x_0,x_1,x_2)).
\end{align*}

Now, let $\lambda$ and $n(\lambda)$ as in Definition \ref{vetro}, then for $n\geq n(\lambda)$ and using the fact that the geometric mean of non-negative real numbers is at most their arithmetic mean, it follows that

\begin{align*}
F(G(x_n,x_{m},x_{l}))& \leq \sum_{k=n}^{l-2}\left[\frac{1}{k}\left(\sum_{i=1}^{k}r_i\right)\right]^kF(G(x_0,x_1,x_2))\\
& =\left(\sum_{k=n}^{l-2} \alpha^k\right) F(G(x_0,x_1,x_2))\\
& \leq \frac{\alpha^n}{1-\alpha}F(G(x_0,x_1,x_2)).
\end{align*}

As $n\to \infty$, we deduce that $G(x_n,x_{m},x_{l}) \to 0.$ Thus $(x_n)$ is a $G$-Cauchy sequence.
and since $X$ is complete there exists $u \in X$ such that $(x_n)$ $G$-converges to $u$.

Moreover, for any positive integers $k, l$, we have
\begin{align*}
F(G(x_n, T_ku,T_lu))& = F(G(T_nx_{n-1},T_ku,T_lu)) \\
& \leq F \left((\tensor*[_{l}]{\Theta}{_{n,k}})\left[G(x_{n-1},T_nx_{n-1},T_nx_{n-1})+\frac{1}{2}[(G(u,T_ku,T_ku)+G(u,T_lu,T_lu)]\right] \right) \nonumber \\
& + F((\tensor*[_{l}]{\Delta}{_{n,k}})G(x_{n-1},u,u)).
\end{align*}

Letting $n\to \infty$, and using property (G3) we obtain

\begin{align*}
F(G(u, T_ku,T_lu))  \leq  (\tensor*[_{l}]{\Theta}{_{n,k}})^s F(G(u, T_ku,T_lu)),
\end{align*}

and this is a contradiction, unless $ u=T_ku=T_lu$, since $\tensor*[_{l}]{\Theta}{_{n,k}}<1$. Then $u$ is a common fixed point of $\{T_n\}$.

\vspace*{0.3cm}

\underline{Claim 2:} $u$ is the unique common fixed point of $\{T_m\}$.

Finally, we prove the uniqueness of the
common fixed point $u$. To this aim, let us suppose that $v$ is another common fixed
point of $\{T_m\}$, that is, $T_m(v) = v, \ \forall m\geq 1$. Then, using \eqref{eq1} again, we have

\[
F(G(u, v,v))= F(G(T_mu, T_mv,T_mv))\leq    (\tensor*[_{m}]{\Delta}{_{m,m}})^sF(G(u,v,v)),
\]
which yields $u=v$, since $\tensor*[_{m}]{\Delta}{_{m,m}}<1$. So, $u$ is the unique common fixed point of $\{T_m\}$.
\end{proof}


\begin{theorem}\label{follow1}

Let $(X,G)$ be a complete $G$-metric space and $\{T_n\}$ be a family of self mappings on $X$ such that

\begin{align}\label{followeq1}
F(G(T_i^px,T_j^py,T_k^pz)) \ \leq  &  
  F\left((\tensor*[_{k}]{\Theta}{_{i,j}})\left[G(x,T_i^px,T_i^px)+\frac{1}{2}[(G(y,T_j^py,T_j^py)+G(z,T_k^pz,T_k^pz)]\right] \right) \nonumber \\
 & + F((\tensor*[_{k}]{\Delta}{_{i,j}})G(x,y,z))
 \end{align}

for all $x,y,z\in X$ with $x\neq y$, $0\leq \tensor*[_{k}]{\Theta}{_{i,j}}, \tensor*[_{k}]{\Delta}{_{i,j}}<1; i,j,k=1,2,\cdots, $ some positive integer $p$, and some $F \in \Phi,$ homogeneous with degree $s$. If 

\[ \sum_{i=1}^{\infty}\left[\frac{[(\tensor*[_{i+2}]{\Theta}{_{i,i+1}})^s+(\tensor*[_{i+2}]{\Delta}{_{i,i+1}})^s]}{1-(\tensor*[_{i+2}]{\Theta}{_{i,i+1}})^s}\right]
\]

is an $\alpha$-series, then $\{T_n\}$ have a unique common fixed point in $X$.

\end{theorem}

\begin{proof}
It follows form Theorem \ref{common1}, that the family $\{T_n^p\}$ have a unique common fixed point $x^*$. Now for any positive integers $i, j, i\neq j$,
\[
T_i(x^*)= T_iT_i^p(x^*)= T_i^pT_i(x^*) \text{ and } T_j(x^*)= T_jT_j^p(x^*)= T_j^pT_j(x^*),
\]
i.e. $T_i(x^*)$ and $T_j(x^*)$ are also fixed points for $T_i^p$ and  $T_j^p$\footnote{Remember that any fixed point of $T_i^p$ is a fixed point of $T_j^p$ for $i\neq j$, Cf. Theorem \ref{common1}.}. Since the common fixed point of $\{T_n^p\}$ is unique, we deduce that 
\[
x^* = T_i(x^*) = T_j(x^*) \text{ for all } i.
\]

\end{proof}

The next result, corollary of Theorem \ref{common1}, corresponds to the result presented by Vetro \cite[Theorem 2.1]{v}.

\begin{corollary}(Compare \cite[Theorem 2.1]{v} )\label{coro1}
Let $(X,G)$ be a complete $G$-metric space and $\{T_n\}$ be a family of self mappings on $X$ such that

\begin{align}\label{eqcor1}
G(T_ix,T_jy,T_kz) \ \leq  &  
  (\tensor*[_{k}]{\Theta}{_{i,j}})[G(x,T_ix,T_ix)+(G(y,T_jy,T_jz)] \nonumber \\
 & + (\tensor*[_{k}]{\Delta}{_{i,j}})G(x,y,z))
\end{align}

for all $x,y,z\in X$ with $x\neq y$, $0\leq \tensor*[_{k}]{\Theta}{_{i,j}}, \tensor*[_{k}]{\Delta}{_{i,j}}<1, i,j,k=1,2,\cdots .$  If 

\[ \sum_{i=1}^{\infty}\left[\frac{[(\tensor*[_{i+2}]{\Theta}{_{i,i+1}})+(\tensor*[_{i+2}]{\Delta}{_{i,i+1}})]}{1-(\tensor*[_{i+2}]{\Theta}{_{i,i+1}})}\right]
\]

is an $\alpha$-series, then $\{T_n\}$ have a unique common fixed point in $X$.

\end{corollary}

\begin{proof}
In Theorem \ref{common1}, take $F=Id_{[0,\infty)}\footnote{The identity map on $[0,\infty)$}$, $j=k$ and $y=z$.

\end{proof}

\section{Second generalizations results}

The next generalisation is that of \cite[Theorem 2.1]{abbas}, the main result of Abbas at al. Instead of considering three maps, we consider a family of maps like in the previous case. Moreover, to show the reader that $\lambda$-sequences do not add to the complexity of the problem, we shall use them in the next statement.

\begin{theorem}\label{common1ref}

Let $X$ be a complete $G$-metric space $(X,G)$ and $\{T_n\}$ be a sequence of self mappings on $X$.  Assume that there exist three sequences $(a_n)$, $(b_n)$ and $(c_n)$ of elements of $X$ such that

\begin{align}\label{condcomon1}
G(T_ix,T_jy,T_kz) \ & \leq (\tensor*[_{k}]{\Delta}{_{i,j}})G(x,y,z) + (\tensor*[_{k}]{\Theta}{_{i,j}})[G(T_ix,x,x)+G(y,T_jy,y)+G(z,z,T_kz)]  \nonumber \\
&+  (\tensor*[_{k}]{\Lambda}{_{i,j}})[G(T_ix,y,z)+G(x,T_jy,z)+G(x,y,T_kz)],
\end{align}

for all $x,y,z\in X$ with $0\leq \tensor*[_{k}]{\Delta}{_{i,j}}+3 (\tensor*[_{k}]{\Theta}{_{i,j}})+4(\tensor*[_{k}]{\Lambda}{_{i,j}}) <1/2 , \ i,j,k = 1,2,\cdots ,$ where $\tensor*[_{k}]{\Delta}{_{i,j}}=G(a_i,a_j,a_k)$, $\tensor*[_{k}]{\Theta}{_{i,j}}=G(b_i,b_j,b_k)$ and $\tensor*[_{k}]{\Lambda}{_{i,j}}=G(c_i,c_j,c_k)$. If the sequence $(r_i)$ where

$$r_i = \left[\frac{[(\tensor*[_{i+2}]{\Delta}{_{i,i+1}})+2(\tensor*[_{i+2}]{\Theta}{_{i,i+1}})+3(\tensor*[_{i+2}]{\Lambda}{_{i,i+1}})]}{1-(\tensor*[_{i+2}]{\Theta}{_{i,i+1}})- (\tensor*[_{i+2}]{\Lambda}{_{i,i+1}})}\right]$$

is a non-increasing $\lambda$-sequence of $\mathbb{R}^+$ endowed with the $\max\footnote{The max metric $m$ refers to $m(x,y)=\max\{x,y\}$ }$ metric, then $\{T_n\}$ have a unique common fixed point in $X$. Moreover, any fixed point of $T_i$ is a fixed point of $T_j$ for $i\neq j$.

\end{theorem}

\begin{proof}
We will proceed in two main steps.

\vspace{0.3cm}

\underline{Claim 1:}  Any fixed point of $T_i$ is also a fixed point of $T_j$ and $T_k$ for $i\neq j\neq k\neq i$.

Assmue that $x^*$ is a fixed point of $T_i$ and suppose that $T_jx^*\neq x^*$ and $T_kx^*\neq x^*$. Then 

\begin{align*}
G(x^*,T_jx^*,T_kx^*) &= G(T_ix^*,T_jx^*,T_kx^*) \\
                 & \leq  (\tensor*[_{k}]{\Delta}{_{i,j}})G(x^*,x^*,x^*) 
                  + (\tensor*[_{k}]{\Theta}{_{i,j}})[G(T_ix^*,x^*,x^*) + G(x^*,T_jx^*,x^*)+ G(x^*,x^*,T_kx^*) ]\\
                 & + (\tensor*[_{k}]{\Lambda}{_{i,j}})[G(T_ix^*,x^*,x^*)+G(x^*,T_jx^*,x^*)+ G(x^*,x^*,T_kx^*) ] \\
                & \leq [ (\tensor*[_{k}]{\Theta}{_{i,j}}) + (\tensor*[_{k}]{\Lambda}{_{i,j}})][G(x^*,T_jx^*,T_kx^*)+G(x^*,T_jx^*,T_kx^*)]\\
                 & \leq [ (2\tensor*[_{k}]{\Theta}{_{i,j}}) + (2\tensor*[_{k}]{\Lambda}{_{i,j}})][G(x^*,T_jx^*,T_kx^*) ],
\end{align*}

which is a contradiction unless $T_ix^* = x^* =T_jx^*=T_kx^*.$

\vspace{0.3cm}

\underline{Claim 2:}

For any $x_0\in X$, we construct the sequence $(x_n)$ by setting $x_n= T_n(x_{n-1}),\ n=1,2,\cdots .$ We assume without loss of generality that $x_n\neq x_{m}$ for all $n\neq m$.
Using \eqref{condcomon1}, we obtain

\begin{align*}
G(x_1,x_2,x_3) & = G(T_1x_0,T_2x_1,T_3x_2) \\
               & \leq (\tensor*[_{3}]{\Delta}{_{1,2}})G(x_0,x_1,x_2) + 
               (\tensor*[_{3}]{\Theta}{_{1,2}})[G(x_1,x_0,x_0)+G(x_1,x_2,x_1)+G(x_2,x_2,x_3)]\\
               &+ (\tensor*[_{3}]{\Lambda}{_{1,2}})[ G(x_1,x_1,x_2)+ G(x_0,x_2,x_2)+G(x_0,x_1,x_3)].
\end{align*}


By property (G3), one can write

\begin{align*}
G(x_1,x_2,x_3) & = G(T_1x_0,T_2x_1,T_3x_2) \\
               & \leq (\tensor*[_{3}]{\Delta}{_{1,2}})G(x_0,x_1,x_2) + 
               (\tensor*[_{3}]{\Theta}{_{1,2}})[G(x_1,x_0,x_2)+G(x_1,x_2,x_0)+G(x_1,x_2,x_3)]\\
               &+ (\tensor*[_{3}]{\Lambda}{_{1,2}})[ G(x_1,x_0,x_2)+ G(x_0,x_1,x_2)+G(x_0,x_1,x_3)]
\end{align*}

Again since

\[
G(x_0,x_1,x_3) \leq G(x_0,x_2,x_2)+ G(x_2,x_1,x_3)\leq  G(x_0,x_1,x_2)+ G(x_2,x_1,x_3),
\]

we obtain, 

\begin{align*}
G(x_1,x_2,x_3) & = G(T_1x_0,T_2x_1,T_3x_2) \\
               & \leq (\tensor*[_{3}]{\Delta}{_{1,2}})G(x_0,x_1,x_2) + 
               (\tensor*[_{3}]{\Theta}{_{1,2}})[G(x_1,x_0,x_2)+G(x_1,x_2,x_0)+G(x_1,x_2,x_3)]\\
               &+ (\tensor*[_{3}]{\Lambda}{_{1,2}})[ G(x_1,x_0,x_2)+ G(x_0,x_1,x_2)+G(x_0,x_1,x_2)+ G(x_2,x_1,x_3)],
\end{align*}

that is 
\[
[1-(\tensor*[_{3}]{\Theta}{_{1,2}})- (\tensor*[_{3}]{\Lambda}{_{1,2}})]G(x_1,x_2,x_3) \leq [(\tensor*[_{3}]{\Delta}{_{1,2}})+2(\tensor*[_{3}]{\Theta}{_{1,2}})+3(\tensor*[_{3}]{\Lambda}{_{1,2}})]G(x_0,x_1,x_2).
\]

Hence

\[
G(x_1,x_2,x_3) \leq \frac{[(\tensor*[_{3}]{\Delta}{_{1,2}})+2(\tensor*[_{3}]{\Theta}{_{1,2}})+3(\tensor*[_{3}]{\Lambda}{_{1,2}})]}{1-(\tensor*[_{3}]{\Theta}{_{1,2}})- (\tensor*[_{3}]{\Lambda}{_{1,2}})}G(x_0,x_1,x_2).
\]


Also we get

\begin{align*}
G(x_2,x_3,x_4) &\leq \frac{[(\tensor*[_{4}]{\Delta}{_{2,3}})+2(\tensor*[_{4}]{\Theta}{_{2,3}})+3(\tensor*[_{4}]{\Lambda}{_{2,3}})]}{1-(\tensor*[_{4}]{\Theta}{_{2,3}})- (\tensor*[_{4}]{\Lambda}{_{2,3}})}G(x_1,x_2,x_3)\\
            &\leq  \left[ \frac{[(\tensor*[_{4}]{\Delta}{_{2,3}})+2(\tensor*[_{4}]{\Theta}{_{2,3}})+3(\tensor*[_{4}]{\Lambda}{_{2,3}})]}{1-(\tensor*[_{4}]{\Theta}{_{2,3}})- (\tensor*[_{4}]{\Lambda}{_{2,3}})}\right] \left[ \frac{[(\tensor*[_{3}]{\Delta}{_{1,2}})+2(\tensor*[_{3}]{\Theta}{_{1,2}})+3(\tensor*[_{3}]{\Lambda}{_{1,2}})]}{1-(\tensor*[_{3}]{\Theta}{_{1,2}})- (\tensor*[_{3}]{\Lambda}{_{1,2}})}\right]G(x_0,x_1,x_2).
\end{align*}

Repeating the above reasoning, we obtain

\vspace*{0.3cm}

\[
G(x_{n},x_{n+1},x_{n+2}) \leq \prod\limits_{i=1}^n \left[\frac{[(\tensor*[_{i+2}]{\Delta}{_{i,i+1}})+2(\tensor*[_{i+2}]{\Theta}{_{i,i+1}})+3(\tensor*[_{i+2}]{\Lambda}{_{i,i+1}})]}{1-(\tensor*[_{i+2}]{\Theta}{_{i,i+1}})- (\tensor*[_{i+2}]{\Lambda}{_{i,i+1}})}\right]G(x_0,x_1,x_2)
\]

If we set 

\vspace*{0.5cm}

$$r_i  = \left[\frac{[(\tensor*[_{i+2}]{\Delta}{_{i,i+1}})+2(\tensor*[_{i+2}]{\Theta}{_{i,i+1}})+3(\tensor*[_{i+2}]{\Lambda}{_{i,i+1}})]}{1-(\tensor*[_{i+2}]{\Theta}{_{i,i+1}})- (\tensor*[_{i+2}]{\Lambda}{_{i,i+1}})}\right], $$

we have that 

$$G(x_{n},x_{n+1},x_{n+2}) \leq \left[ \prod\limits_{i=1}^n r_i \right] G(x_0,x_1,x_2). $$



Therefore, for all $l>m>n>2$

\begin{align*}
G(x_n,x_{m},x_{l}) & \leq G(x_{n},x_{n+1},x_{n+1})+  G(x_{n+1},x_{n+2},x_{n+2})  \\
 & + \cdots + G(x_{l-1},x_{l-1},x_{l}) \\
 & \leq G(x_{n},x_{n+1},x_{n+2})+  G(x_{n+1},x_{n+2},x_{n+3})  \\
 & + \cdots + G(x_{l-2},x_{l-1},x_{l}),
\end{align*}


and

\begin{align*}
G(x_n,x_{m},x_{l})& \leq \left(\left[ \prod\limits_{i=1}^n r_i \right] + 
\left[ \prod\limits_{i=1}^{n+1} r_i \right]+ \cdots + \left[ \prod\limits_{i=1}^{l-2} r_i \right]\right) G(x_0,x_1,x_2) \\
& = \sum_{k=0}^{l-n-2}\left[\prod\limits_{i=1}^{n+k}r_i\right]G(x_0,x_1,x_2)\\
& =\sum_{k=n}^{l-2}\left[\prod\limits_{i=1}^{k}r_i\right]G(x_0,x_1,x_2).
\end{align*}

Now, let $\lambda$ and $n(\lambda)$ as in Definition \ref{vetro}, then for $n\geq n(\lambda)$ and using the fact that the geometric mean of non-negative real numbers is at most their arithmetic mean, it follows that

\begin{align*}
G(x_n,x_{m},x_{l})& \leq \sum_{k=n}^{l-2}\left[\frac{1}{k}\left(\sum_{i=1}^{k}r_i\right)\right]^kG(x_0,x_1,x_2)\\
& =\left(\sum_{k=n}^{l-2} \alpha^k\right) G(x_0,x_1,x_2)\\
& \leq \frac{\alpha^n}{1-\alpha}G(x_0,x_1,x_2).
\end{align*}

As $n\to \infty$, we deduce that $G(x_n,x_{m},x_{l}) \to 0.$ Thus $(x_n)$ is a $G$-Cauchy sequence.
Moreover, since $X$ is complete there exists $u \in X$ such that $(x_n)$ $G$-converges to $u$.


%
%
%
%
%
%
%
%
%
%
%
%

If there exists $n_0$ such that $T_{n_0}u=u$, then by the claim 1, the proof of existence is complete.

Otherwise for any positive integers $k, l$, we have
\begin{align*}
G(x_n, T_ku,T_lu)  & = G(T_nx_{n-1},T_ku,T_lu) \\
                   & \leq (\tensor*[_{l}]{\Delta}{_{n,k}})G(x_{n-1},u,u) + (\tensor*[_{l}]{\Theta}{_{n,k}})[G(T_nx_{n-1},x_{n-1},x_{n-1})+G(u,T_ku,u)+G(u,u,T_lu)]\\
                   &+ (\tensor*[_{l}]{\Lambda}{_{n,k}})[G(T_nx_{n-1},u,u)+G(x_{n-1},T_ku,u)+G(x_{n-1},u,T_lu)]
\end{align*}

\vspace*{0.5cm}

Letting $n\to \infty$, and using property (G3) we obtain

\begin{align*}
G(u, T_ku,T_lu)  & \leq (\tensor*[_{l}]{\Theta}{_{n,k}})[G(u,T_ku,u)+G(u,u,T_lu)] \\
& + (\tensor*[_{l}]{\Lambda}{_{n,k}})[G(u,T_ku,u)+G(u,u,T_lu)]\\
& \leq  [ (2\tensor*[_{k}]{\Theta}{_{i,j}}) + (2\tensor*[_{k}]{\Lambda}{_{i,j}})][G(u,T_ku,T_lu)+G(u,T_ku,T_lu)]
\end{align*}
and this is a contradiction, unless $ u=T_ku=T_lu$.

Finally, we prove the uniqueness of the
common fixed point $u$. To this aim, let us suppose that $v$ is another common fixed
point of ${T_m}$, that is, $T_m(v) = v, \ \forall m\geq 1$. Then, using \ref{condcomon1}, we have

\[
G(u, v,v)= G(T_nu, T_kv,T_lv)   \leq (\tensor*[_{l}]{\Delta}{_{n,k}})G(u,v,v)+3 (\tensor*[_{l}]{\Lambda}{_{n,k}}) G(u,v,v),
\]
which yields $u=v$. So, $u$ is the unique common fixed point of $\{T_m\}$.
\end{proof}

Following the same lines of the proof of Theorem \ref{follow1}, one can prove the next theorem.

\begin{theorem}\label{common1refpower}

Let $X$ be a complete $G$-metric space $(X,G)$ and $\{T_n\}$ be a sequence of self mappings on $X$.  Assume that there exist three sequences $(a_n)$, $(b_n)$ and $(c_n)$ of elements of $X$ such that

\begin{align}\label{condcomon1power}
G(T_i^px,T_j^py,T_k^pz) \ & \leq (\tensor*[_{k}]{\Delta}{_{i,j}})G(x,y,z) + (\tensor*[_{k}]{\Theta}{_{i,j}})[G(T_i^px,x,x)+G(y,T_j^py,y)+G(z,z,T_k^pz)]  \nonumber \\
&+  (\tensor*[_{k}]{\Lambda}{_{i,j}})[G(T_i^px,y,z)+G(x,T_j^py,z)+G(x,y,T_k^pz)],
\end{align}

for all $x,y,z\in X$ with $0\leq \tensor*[_{k}]{\Delta}{_{i,j}}+3 (\tensor*[_{k}]{\Theta}{_{i,j}})+4(\tensor*[_{k}]{\Lambda}{_{i,j}}) <1/2 , \ i,j,k = 1,2,\cdots ,$ some positive integer $p$, where $\tensor*[_{k}]{\Delta}{_{i,j}}=G(a_i,a_j,a_k)$, $\tensor*[_{k}]{\Theta}{_{i,j}}=G(b_i,b_j,b_k)$ and $\tensor*[_{k}]{\Lambda}{_{i,j}}=G(c_i,c_j,c_k)$. If the sequence $(r_i)$ where

$$r_i = \left[\frac{[(\tensor*[_{i+2}]{\Delta}{_{i,i+1}})+2(\tensor*[_{i+2}]{\Theta}{_{i,i+1}})+3(\tensor*[_{i+2}]{\Lambda}{_{i,i+1}})]}{1-(\tensor*[_{i+2}]{\Theta}{_{i,i+1}})- (\tensor*[_{i+2}]{\Lambda}{_{i,i+1}})}\right]$$

is a non-increasing $\lambda$-sequence of $\mathbb{R}^+$ endowed with the $\max\footnote{The max metric $m$ refers to $m(x,y)=\max\{x,y\}$ }$ metric, then $\{T_n\}$ have a unique common fixed point in $X$. Moreover, any fixed point of $T_i$ is a fixed point of $T_j$ for $i\neq j$.

\end{theorem}

\vspace*{0.5cm}

The next result, corollary of Theorem \ref{common1ref}, corresponds to the result presented by Abbas \cite[Theorem 2.1]{abbas}.

\begin{corollary}\label{common1refabbas}

Let $X$ be a complete $G$-metric space $(X,G)$, $f,g,h$  mappings on $X$.  Assume that there exist three positive reals $a,bc$ such that

\begin{align}\label{condcomon1abbas}
G(fx,gy,hz) \ & \leq a G(x,y,z) + b[G(fx,x,x)+G(y,gy,y)+G(z,z,hz)]  \nonumber \\
&+  c[G(fx,y,z)+G(x,gy,z)+G(x,y,hz)],
\end{align}
for all $x,y,z\in X$ with $0\leq a+3 b+4c <1 .$ 
Then $f,g,h$ have a unique common fixed point in $X$. Moreover, any fixed point of $f$ is a fixed point of $g$ and $h$ and conversely.

\end{corollary}

\begin{proof}

In Theorem \ref{common1ref}, take $T_1=f, T_2=g,T_3=h$. Also set 
\[
\tensor*[_{3}]{\Delta}{_{1,2}} = a,\ \tensor*[_{3}]{\Theta}{_{1,2}}=b, \  \tensor*[_{3}]{\Lambda}{_{1,2}}=c.
\]
Hence, we have:

\begin{align*}
0\leq a+3 b+4c <1/2  & \Longrightarrow 0\leq  a+3 b+4c <1 \\
               & \Longleftrightarrow 0\leq  r_i = r = \left[\frac{a+2b+3c}{1-b- c}\right]<1. \\
\end{align*}
The sequence $r_i = r$ is constant, so in Definition \ref{vetro}, if we choose $\lambda = \frac{1}{2}$ and $n(\lambda)=1$, it is clear that $\sum_{i=1}^\infty r_i$ is an $\alpha$-series. Indeed, since
\[
\left[\frac{a+2b+3c}{1-b- c}\right] < a+3 b+4c <\frac{1}{2},  
\]
therefore, for any $L\geq n(\lambda)+1=1+1=2,$

\[
\sum_{i=1}^{L-1} r_i = \sum_{i=1}^{L-1} r < \frac{1}{2}(L-1) \leq  \frac{1}{2}L.
\]
\end{proof}


We conclude this manuscript with the following result, whose proof is straightforward, following the steps of the proofs of the earliest results.

\begin{theorem}\label{lambdafin}
Let $X$ be a complete $G$-metric space $(X,G)$ and $\{T_n\}$ be a sequence of self mappings on $X$.  Assume that there exist three sequences $(a_n)$, $(b_n)$ and $(c_n)$ of elements of $X$ such that

\begin{align}\label{conditionfin}
F[G(T_i^px,T_j^py,T_k^pz)] \ & \leq F[(\tensor*[_{k}]{\Delta}{_{i,j}})G(x,y,z) + (\tensor*[_{k}]{\Theta}{_{i,j}})[G(T_i^px,x,x)+G(y,T_j^py,y)+G(z,z,T_k^pz)]  \nonumber \\
&+  (\tensor*[_{k}]{\Lambda}{_{i,j}})[G(T_i^px,y,z)+G(x,T_j^py,z)+G(x,y,T_k^pz)]],
\end{align}

for all $x,y,z\in X$ with $0\leq (\tensor*[_{k}]{\Delta}{_{i,j}})^s+3 (\tensor*[_{k}]{\Theta}{_{i,j}})^s+4(\tensor*[_{k}]{\Lambda}{_{i,j}})^s <1/2 , \ i,j,k = 1,2,\cdots ,$ some positive integer $p$ and some $F \in \Phi,$ homogeneous with degree $s$, where $\tensor*[_{k}]{\Delta}{_{i,j}}=G(a_i,a_j,a_k)$, $\tensor*[_{k}]{\Theta}{_{i,j}}=G(b_i,b_j,b_k)$ and $\tensor*[_{k}]{\Lambda}{_{i,j}}=G(c_i,c_j,c_k)$. If the sequence $(r_i)$ where

$$r_i = \left[\frac{[(\tensor*[_{i+2}]{\Delta}{_{i,i+1}})^s+2(\tensor*[_{i+2}]{\Theta}{_{i,i+1}})^s+3(\tensor*[_{i+2}]{\Lambda}{_{i,i+1}})^s]}{1-(\tensor*[_{i+2}]{\Theta}{_{i,i+1}})^s- (\tensor*[_{i+2}]{\Lambda}{_{i,i+1}})^s}\right]$$

is a non-increasing $\lambda$-sequence of $\mathbb{R}^+$ endowed with the $\max\footnote{The max metric $m$ refers to $m(x,y)=\max\{x,y\}$ }$ metric, then $\{T_n\}$ have a unique common fixed point in $X$. Moreover, any fixed point of $T_i$ is a fixed point of $T_j$ for $i\neq j$.

\end{theorem}

In addition to the examples provided by Abbas and Vetro, illustrations of all the above results can be read in \cite[Example 2.5]{Gaba1} and \cite[Example 2.8]{Gaba4}.


\section*{Conflict of interests }

The author declares that there is no conflict
of interests regarding the publication of this article.

\section*{ Acknowledgments.} This work was carried out with financial support from the government of Canada’s International
Development Research Centre (IDRC), and within the framework of the AIMS Research for Africa
Project.

\bibliographystyle{amsplain}

\end{document}